\documentclass[draft]{nzjm}
\usepackage{amsmath,amsthm,amsfonts,amssymb,amsbsy,latexsym,graphicx}
\usepackage{mathrsfs}

\def\currentvolume{**}
\def\currentyear{****}
\def\pagerange{**--**}
\newtheorem{theorem}{Theorem}[section]
\newtheorem{corollary}[theorem]{Corollary}
\newtheorem{lemma}[theorem]{Lemma}

\theoremstyle{definition}
\newtheorem{remark}[theorem]{Remark}

\numberwithin{equation}{section}

\newcommand{\racion}[2]{\mbox{\small$\frac{{#1}}{{#2}}$}}

\begin{document}

\title[Some generating functions for $q$-polynomials]
{Some generating functions for $q$-polynomials}

\author{Howard S. Cohl}
\author{Roberto S. Costas-Santos}
\author{Tanay V. Wakhare}
\date{20 July, 2017}
\keywords{Basic Hypergeometric Functions, Generating 
Functions, $q$-Polynomials}
\subjclass[2010]{33D45; 33C20}
\begin{abstract}
We obtain $q$-analogues of the Sylvester, Ces\`{a}ro, Pasternack, and
Bateman polynomials. We also derive generating functions for these polynomials.
\end{abstract}
\maketitle
\section{Introduction}
Trying to construct inverse Laplace transforms in order to 
better understand some radiation and conduction problems, 
H. Bateman introduced in \cite{bat1} the polynomial
\[
{\mathscr Z}_n(z)={}_2F_2\left(\!\!\begin{array}{c}-n, n+1\\
1, 1\end{array}\!;z\right),
\]
which we refer to as the Bateman-${\mathscr Z}$ polynomial. 
By using \cite[Theorem 48]{rai}, we obtain the generating function
\[
\sum_{n=0}^\infty {\mathscr Z}_n(z)\, t^n=\frac 1{1-t}\,{}_1F_1
\left(\!\!\begin{array}{c} \racion 12\\ 1\end{array}\!;
\dfrac{-4zt}{(1-t)^2} \right). 
\]
Using the above information, the Bateman polynomials are 
defined  as \cite[p. 25]{bat2}
\[
{\mathscr B}_n(z)={}_3F_ 2\left(\begin{array}{c}-n, n+1, 
\racion{z+1}2\\1, 1\end{array}\!;1\right).
\]
He also obtained the generating functions
\[
\sum_{n=0}^\infty {\mathscr B}_n(z)\, t^n=\frac 1{1-t}\,{}_2F_1\left(\!\!
\begin{array}{c} \racion 12, \racion{z+1}2\\ 1\end{array}\!;\dfrac{-4t}
{(1-t)^2}\right), \]
\[
\sum_{n=0}^\infty \left({\mathscr B}_n(z-2)-{\mathscr B}_n(z)\right)\, t^n=
\frac {2t}{(1-t)^3}\,{}_2F_1\left(\!\!\begin{array}{c} \racion 32,
\racion{z+1}2\\ 2 \end{array}\!;\dfrac{-4t}{(1-t)^2}
\right). \]
\begin{lemma} \label{lem1}
Let $t\in \mathbb C$. Then the following relation holds:
\begin{equation} \label{1.1}
\sum_{n=0}^\infty {\mathscr B}_m(-2n-1)\frac{t^n}
{n!}=e^t {\mathscr Z}_m(-t).
\end{equation}
\end{lemma} 
The Bateman polynomial $\mathscr B_n$ was generalized 
by Pasternack in \cite{pas}. 
He defines the Pasternack polynomial $\mathscr B_n^m$ as
\[
{\mathscr B}_n^m(z)={}_3F_2\left(\!\!\begin{array}{c}-n, n+1, 
\racion{z+m+1}2\\1, m+1\end{array}\!;1\right),
\]
for $m\in{\mathbb C}\backslash \{-1\}$. 
The Pasternack polynomials reduce to the Bateman polynomials 
when $m=0$. 

Further information regarding such polynomials and their connection 
with (classical) orthogonal polynomials can be found in \cite{koe}.
Indeed, we can write the Pasternack polynomials in terms of the 
continuous Hahn polynomials \cite[(9.4.1)]{kost} as follows \cite[p. 893]{koe}:
\[
\mathscr B_n^m(z)= {\frac 1{i^n (m+1)_n}}\, p_n\left( \frac{-iz}2; 
\frac {1+m}2, \frac{1-m}2,\frac{1-m}2, \frac{1+m}2\right).
\]

We also consider the Sylvester polynomials, defined as 
(see \cite[p. 185]{srma})
\[
\varphi_n(z)=\frac{z^n}{n!}\, {}_2F_0\left(\begin{array}{c} 
-n, z\\ \--\end{array}\!;\frac{-1}z\right).
\]
Notice that we also can write the Sylvester polynomials in terms 
of (classical) orthogonal polynomials \cite{kost}
\[
\varphi_n(z)=(-1)^n L_n^{(-z-n)}(z)=\frac{z^n}{n!}\, C_n(-z;z).
\]
Here $L_n^{(\alpha)}$ and $C_n$ represent the Laguerre and 
Charlier polynomials, respectively.
It is also known that the Sylvester polynomials satisfy the 
generating functions
\[
\sum_{n=0}^\infty \varphi_n(z)\, t^n=\frac {e^{z t}}{(1-t)^z},
\]
\[
\sum_{n=0}^\infty (\lambda)_n\,\varphi_n(z)\, t^n = \frac {1}
{(1-zt)^\lambda}\, {}_2F_0\left(\!\!\begin{array}{c} \lambda, z\\ 
\--\end{array}\!;\frac{t}{1-zt}\right).
\]

The Ces\`aro polynomials are defined as \cite[p. 449]{srma}
\[
g^{(s)}_n(z)=\frac{(1+s)_n}{n!}\, {}_2F_1\left(
\!\!\begin{array}{c}-n, 1\\ -s-n\end{array}\!;z \right).
\]
Observe that this family can be written in terms of Jacobi
polynomials \cite[p. 449]{srma} as
\[
g^{(s)}_n(z)=P^{(s+1,-s-n-1)}_n(2z-1).
\]
Furthermore, they satisfy the generating functions
\[
\sum_{n=0}^\infty g_n^{(s)}(z)\, t^n=(1-t)^{-s-1}(1-zt)^{-1},
\]
\[
\sum_{n=0}^\infty \binom {n+\ell} \ell \, g^{(s)}_{n+\ell}(z)\, t^n=
(1-t)^{-s-1-\ell}(1-zt)^{-1} g^{(s)}_\ell\left(\frac{z(1-t)}{1-zt}\right),
\]
where $\ell=0,1,2,...$.

The aim of this paper is to obtain $q$-analogues of 
all these families of polynomials as well as $q$-analogues 
of the generating functions stated above. 
The structure of this paper is as follows. In Section 2
we give some preliminaries on $q$-calculus and we define 
some $q$-analogues of Bateman-$\mathscr Z$, Bateman, 
Sylvester,  Pasternack, and Ces\`aro polynomials.
In Section 3 we state and derive $q$-analogues of the given
generating functions associated with the $q$-Bateman-$\mathscr Z$, 
$q$-Bateman, $q$-Sylvester, $q$-Pasternack, and 
$q$-Ces\`aro polynomials.
\section{Some $q$-Calculus. The Definitions.}
Let $q\in \mathbb C$, $0<|q|<1$. A $q$-analogue of the hypergeometric 
series ${}_pF_r$ is the basic hypergeometric series \cite{gara}
\[
{}_r\phi_s\left(\!\!\begin{array}{c} a_1, a_2, \dots, a_r\\ b_1, b_2, \dots, b_s
\end{array}\!;q, z\right)=\sum_{k=0}^\infty \frac{(a_1;q)_k (a_2;q)_k \cdots
(a_r;q)_k}{(b_1;q)_k(b_2;q)_k\cdots (b_s;q)_k}\left((-1)^kq^{\racion{k(k-1)}2}
\right)^{1+s-r} \frac{z^k}{(q;q)_k},
\]
where $q\ne 0$ when $r>s+1$, and the $(b_i)$ are such that the 
denominator never vanishes.  
We also need to define some other $q$-analogues, such as the 
$q$-analogue of a 
number $[a]_q$, factorial $[a]_q!$, 
and the Pochhammer symbol (rising factorial), $(a)_n$. 
These $q$-analogues are given as follows:
\[
[a]_q=\frac{1-q^a}{1-q},
\]
\[
[0]_q!=1,\qquad [n]_q!=\prod_{k=1}^n [k]_q,\ n\in \mathbb N,
\]
\[
(a;q)_0=1,\quad (a;q)_n=(1-a)(1-q)\cdots (1-aq^{n-1}),\quad n\in \mathbb N.
\]
We also will need a $q$-analogue of the {\it binomial theorem}
\cite[(1.3.2)]{gara}
\begin{equation}
\label{2.1}
{}_1\phi_0\left(\!\!\begin{array}{c} a\\ \-- \end{array}\!;q, z\right)
=\sum_{k=0}^\infty \frac{(a;q)_k}{(q;q)_k}\, z^k=\frac{(az;q)_\infty}
{(z;q)_\infty},\qquad |z|<1,\ |q|< 1,
\end{equation}
and the $q$-binomial coefficient 
\begin{equation}
\left[\!\!\begin{array}{c}n \\ k 
\end{array}\!\!\right]_q = \frac{(q;q)_n}{(q;q)_{k} (q;q)_{n-k}}.
\label{qbinomcoef}
\end{equation}

Taking into account the previous definitions and results, we define $q$ 
analogues of the above introduced polynomials. We define
\[
\mathscr Z_{n}(z;q)={}_2\phi_2\left(\!\!\begin{array}{c} q^{-n}, q^{n+1}\\
q, q \end{array}\!;q, q^n z\right),
\]
and the $q$-Bateman polynomial as
\[
\mathscr B_{n}(z;q)={}_3\phi_2\left(\!\!\begin{array}{c} q^{-n}, q^{n+1},
q^{\racion{1+z}2}\\ q, q \end{array}\!;q, q^n\right),
\]
the $q$-Pasternack polynomial as
\[
\mathscr B^m_{n}(z;q)={}_3\phi_2\left(\!\!\begin{array}{c} q^{-n}, q^{n+1},
q^{\racion{1+z+m}2}\\ q, q^{m+1} \end{array}\!;q, q^n\right),
\]
the $q$-Sylvester polynomial as
\[
\varphi_{n}(z;q)=\frac{z^n}{(q;q)_n}\,{}_2\phi_0\left(\!\!\begin{array}{c}
q^{-n}, q^{z}\\ \-- \end{array}\!;q, q^n z^{-1}\right),
\]
the $q$-Ces\`aro polynomial as
\[
g^{(s)}_{n}(z;q)=\frac{(q^{s+1};q)_n}{(q;q)_n}\,{}_2\phi_1\left(\!\!
\begin{array}{c}q^{-n}, q\\ q^{-s-n} \end{array}\!;q, z\right).
\]
\begin{remark}
Note that $\mathscr B^0_{n}(z;q)=\mathscr B_{n}(z;q)$.
\end{remark}
\begin{lemma}
The $q$-Ces\`aro polynomial can be written as
\begin{equation} \label{qCP} 
g^{(s)}_{n}(z;q)=\sum_{k=0}^n \left[\!\!\begin{array}{c}k+s \\ s 
\end{array}\!\!\right]_q (zq^s)^{n-k}.
\end{equation}
\end{lemma}
\section{The Generating Functions}
\begin{theorem} \label{thm:3.1} For any $t\in \mathbb C$ small enough, 
the  $q$-Bateman-${\mathscr Z}$ 
and the $q$-Bateman polynomials satisfy the following generating 
functions:
\begin{align}\label{3.1}
\sum_{n=0}^\infty {\mathscr Z}_{n}(z;q)\, t^n=& 
\frac 1{1-t}\,{}_4\phi_5\left(\!\!\begin{array}{c} -q, q^{\racion 12}, 
-q^{\racion 12}, 0,\\ q, qt^{\racion 12},-qt^{\racion 12}, (qt)^{\racion12}, 
-(qt)^{\racion 12}\end{array}\!;q, z t\right),\\[3mm]
\label{3.2} \sum_{n=0}^\infty {\mathscr B}_{n}(z;q)\,t^n=&\frac 
1{1-t}\,{}_5\phi_5 \left(\!\!\begin{array}{c} -q,q^{\racion12}, -q^{\racion
12},  q^{\racion{z+1}2}, 0\\ q, (qt)^{\racion12}, -(qt)^{\racion 12}, 
qt^{\racion12}, -qt^{\racion12} \end{array}\!;q,t\right), 
\\[3mm] \label{3.3} \sum_{n=0}^\infty \Big({\mathscr B}_{n}(z-2;q)-{\mathscr 
B}_{n}(z;q)\Big)\, t^n=&\frac {q^{\racion {z-1}2} (1+q) t}{(t;q)_3}\,{}_5
\phi_5\left(\!\!\begin{array}{c} -q^2, q^{\racion 32}, -q^{\racion 32},  
q^{\racion{z+1}2}, 0\\ q^2, (q^3t)^{\racion 12},-(q^3t)^{\racion 12}
,q^2t^{\racion 12}, -q^2t^{\racion 12} \end{array}\!;q, qt\right).
\end{align}
\end{theorem}
\begin{proof} Let us start by proving identity \eqref{3.2}, namely
\[
\sum_{n=0}^\infty {\mathscr B}_{n}(z;q)\,t^n=
\sum_{n=0}^\infty \sum_{k=0}^n \frac{(q^{-n};q)_k(q^{n+1};q)_k
(q^{\racion{z+1}2};q)_k}{(q;q)^3_k}(q^kt)^n.
\]
By using the identity
\[
(q^{-n};q)_k=\frac{(q;q)_n}{(q;q)_{n-k}}(-1)^k q^{\racion{k(k-1)}2-nk},
\qquad k=0, 1, \dots, n,
\]
we have
\[
\sum_{n=0}^\infty {\mathscr B}_{n}(z;q)\,t^n=\sum_{n=0}^\infty 
\sum_{k=0}^n \frac{(q;q)_{n+k}(q^{\racion{z+1}2};q)_k}
{(q;q)_{n-k}(q;q)^3_k}(-1)^k q^{\racion{k(k-1)}2}t^n.
\]
Next, we rearrange the double summation and set 
$n\mapsto n+k$, obtaining
\begin{align*}
\sum_{n=0}^\infty {\mathscr B}_{n}(z;q)\,t^n=&\sum_{k=0}^\infty
\sum_{n=0}^\infty \frac{(q;q)_{n+2k}(q^{\racion{z+1}2};q)_k}
{(q;q)_{n}(q;q)^3_k}(-1)^k q^{\racion{k(k-1)}2}t^{n+k}\\[3mm]
=& \sum_{k=0}^\infty\frac{(q;q)_{2k}(q^{\racion{z+1}2};q)_k}
{(q;q)^3_k}(-t)^k q^{\racion{k(k-1)}2}
\sum_{n=0}^\infty \frac{(q^{2k+1};q)_{n}}{(q;q)_{n}}t^{n}.
\end{align*}
By using the identity
$$
(a;q)_{2n}=(a^{\racion 12};q)_n(-a^{\racion 12};q)_n((aq)^{\racion12};q)_n
(-(aq)^{\racion12};q)_n,
$$
and \eqref{2.1}, we obtain for the generating function
\begin{align*}
=& \sum_{k=0}^\infty\frac{(-q;q)_k(q^{\racion 12};q)_k(-q^{\racion 12};q)_k
(q^{\racion{z+1}2};q)_k}{(q;q)^2_k}(-t)^k q^{\racion{k(k-1)}2}
\frac{(q^{2k+1}t;q)_{\infty}}{(t;q)_{\infty}} \\[3mm]
=& \sum_{k=0}^\infty\frac{(-q;q)_k(q^{\racion 12};q)_k(-q^{\racion 12};q)_k
(q^{\racion{z+1}2};q)_k}{(t;q)_{2k+1}(q;q)^2_k}(-t)^k q^{\racion{k(k-1)}2}
\\[3mm]=&\frac1{1-t}\sum_{k=0}^\infty\frac{(-q;q)_k(q^{\racion 12};q)_k
(-q^{\racion 12};q)_k(q^{\racion{z+1}2};q)_k}{(qt;q)_{2k}(q;q)^2_k}(-t)^k
q^{\racion{k(k-1)}2}\\ =&\frac1{1-t}\sum_{k=0}^\infty\frac{(-q;q)_k
(q^{\racion 12};q)_k(-q^{\racion 12};q)_k(q^{\racion{z+1}2};q)_k}{(q;q)_k
((qt)^{\racion 12};q)_{k}(-(qt)^{\racion 12};q)_{k}(qt^{\racion 12};q)_{k}
(-qt^{\racion 12};q)_{k}}(-1)^k q^{\racion{k(k-1)}2}\frac{t^k}{(q;q)_k}.
\end{align*}
Hence the identity follows. 
In order to prove identity \eqref{3.3}:
\begin{align*}
\sum_{n=0}^\infty \Big({\mathscr B}_{n}(z-2;q)-{\mathscr B}_{n}(z;q)
\Big)\, t^n=&\sum_{n=0}^\infty\sum_{k=0}^n\frac{(q^{-n};q)_k(q^{n+1};q)_k}
{(q;q)^3_k}q^{nk}\left((q^{\racion{z-1}2};q)_k-(q^{\racion{z+1}2}
;q)_k\right)t^n\\[3mm]=& \sum_{n=0}^\infty\sum_{k=0}^n\frac{(q^{-n};q)_k
(q^{n+1};q)_k}{(q;q)^3_k}q^{nk}(q^{\racion{z+1}2};q)_{k-1}(1-q^k)
(-q^{\racion{z-1}2})t^n \\[3mm]=& \sum_{n=0}^\infty\sum_{k=0}^n\frac{(q;q)_{n+k}
(q^{\racion{z+1}2};q)_{k-1}(-1)^{k-1} q^{\racion{z-1}2} q^{\racion{k(k-1)}2}}
{(q;q)_{n-k}(q;q)_{k-1}(q;q)^2_k}t^n.
\end{align*}
Taking into account that the above expression vanishes at $k=0$ 
we set $k\mapsto k+1$, rearrange the double sum, and set 
$n\mapsto n+k$, yielding
\begin{align*}
=& \sum_{k=0}^\infty\sum_{n=0}^\infty\frac{(q;q)_{n+2k+1}
(q^{\racion{z+1}2};q)_{k}(-1)^{k} q^{\racion{z-1}2} q^{\racion{k(k+1)}2}}{
(q;q)_{n-1}(q;q)_{k}(q;q)_{k+1}(q;q)_{k+1}}\,t^{n+k}.
\end{align*}
Here, again, the series vanishes at $n=0$ so we set $n\mapsto n+1$. 
Applying some basic identities of $q$-Pochhammer symbols, we obtain
\begin{align*}
=&\, (1+q)q^{\racion {z-1}2}t\sum_{k=0}^\infty\sum_{n=0}^\infty
\frac{(q^3;q)_{n+2k}(q^{\racion{z+1}2};q)_{k} q^{\racion{k(k-1)}2}}
{(q;q)_{n}(q;q)_{k}(q^2;q)_{k}(q^2;q)_{k}}\,(-qt)^kt^{n}\\[3mm]
=& (1+q)q^{\racion {z-1}2}t\sum_{k=0}^\infty\frac{(q^3;q)_{2k}
(q^{\racion{z+1}2};q)_{k}}{(q;q)_{k}(q^2;q)_{k}(q^2;q)_{k}}\,
 q^{\racion{k(k-1)}2} (-qt)^k\sum_{n=0}^\infty \frac{(q^{3+2k};q)_n}
{(q;q)_{n}}\,t^{n}.
\end{align*}
Applying again \eqref{2.1} and simplifying 
one has
\begin{align*}
=& (1+q)q^{\racion {z-1}2}t\sum_{k=0}^\infty\frac{(q^3;q)_{2k}
(q^{\racion{z+1}2};q)_{k}}{(q;q)_{k}(q^2;q)_{k}(q^2;q)_{k}}\,
 q^{\racion{k(k-1)}2} (-qt)^k\frac{1}{(t;q)_{3+2k}}\\[3mm]
=& (1+q)q^{\racion {z-1}2}t\sum_{k=0}^\infty\frac{(q^{\racion 32};q)_{k}
(-q^{\racion 32};q)_{k}(-q^2;q)_{k}(q^{\racion{z+1}2};q)_{k}}{(q;q)_{k}
(q^2;q)_{k}}\, q^{\racion{k(k-1)}2} (-qt)^k\frac{1}{(t;q)_{3+2k}}\\[3mm]
=& \frac{(1+q)}{(t;q)_3}\,q^{\racion {z-1}2}t\sum_{k=0}^\infty 
\frac{(q^{\racion32};q)_{k}(-q^{\racion 32};q)_{k}(-q^2;q)_{k}
(q^{\racion{z+1}2};q)_{k}q^{\racion{k(k-1)}2}}{(q;q)_{k}(q^2;q)_{k}
((tq^3)^{\racion 12};q)_k(-(tq^3)^{\racion 12};q)_k (q^2t^{\racion 12};q)_k
(-q^2t^{\racion 12};q)_k}\,(-qt)^k.
\end{align*}
Therefore the identity follows. Let us prove the generating function 
\eqref{3.3}. We have 
\begin{align*}
\sum_{n=0}^\infty {\mathscr Z}_{n}(z;q)\, t^n=& \sum_{n=0}^\infty 
\sum_{k=0}^n\frac{(q^{-n};q)_k (q^{n+1};q)_k (-1)^k q^{\racion{k(k-1)}2}}{
(q;q)^3_k}\, z^k q^{nk}t^n\\[3mm] = & \sum_{n=0}^\infty \sum_{k=0}^n
\frac{(q;q)_{n+k}\, q^{k(k-1)}}{(q;q)_{n-k}(q;q)^3_k}\, z^kt^n.
\end{align*}
As in the previous identities, we rearrange the double sums and 
set $n\mapsto n+k$ obtaining
\begin{align*}
= & \sum_{k=0}^\infty \sum_{n=0}^\infty \frac{(q;q)_{n+2k}\, q^{k(k-1)}}
{(q;q)_{n}(q;q)^3_k}\, z^kt^{n+k}=\sum_{k=0}^\infty
\frac{(q;q)_{2k}\, q^{k(k-1)}}{(q;q)^3_k} (zt)^k
\sum_{n=0}^\infty \frac{(q^{2k+1};q)_n}{(q;q)_{n}}\,t^n\\[3mm]
=& \sum_{k=0}^\infty \frac{(q;q)_{2k}\, q^{k(k-1)}}{(q;q)^3_k}\,
\frac{(zt)^k}{(t;q)_{2k+1}}=\frac1{1-t}\sum_{k=0}^\infty \frac{(-q;q)_{k}
(q^{\racion 12};q)_k(-q^{\racion 12};q)_k\, q^{k(k-1)}}{(q;q)_k (q;q)_k\
(tq;q)_{2k}}\,(zt)^k\\[3mm] =& \frac1{1-t}\sum_{k=0}^\infty \frac{(-q;q)_{k}
(q^{\racion 12};q)_k(-q^{\racion 12};q)_k\, q^{k(k-1)}}{(q;q)_k (q;q)_k
((tq)^{\racion 12};q)_{k}(-(tq)^{\racion 12};q)_{k}(qt^{\racion 12};q)_{k}
(-qt^{\racion 12};q)_{k}}\,(zt)^k.
\end{align*}
Hence the identity follows.
\end{proof}
\begin{theorem} For any $t\in \mathbb C$ small enough, the following identities hold:
\begin{align}\label{3.11}
\sum_{n=0}^\infty {\mathscr B}_{m}(-2n-1;q)\,\frac{t^n}{(q;q)_n}=
&\frac{1}{(t;q)_\infty}{\mathscr Z}_{m}(tq^{-n};q),
\end{align}
\begin{align}\label{3.9}
\sum_{n=0}^\infty {\mathscr B}^{j}_{m}(-2n-1-j;q)\,t^n = & 
\frac 1{1-t}\, {}_2\phi_2\left(\begin{array}{c} q^{-m}, q^{m+1} \\ 
q^{j+1}, t q \end{array} \!; q; t\right),
\end{align}
\begin{align}\label{3.10}
\sum_{n=0}^\infty {\mathscr B}^{j}_{m}(-2n-1-j;q)\,\frac{(q^\lambda;q)_n}
{(q;q)_n}\, t^n= & \ \frac 1{(t;q)_\lambda} \, {}_3\phi_3\left(
\begin{array}{c} q^{-m}, q^{m+1}, q^\lambda \\ q, q^{j+1}, t q^\lambda 
\end{array} \!; q; t\right), \qquad \lambda \in \mathbb C.
\end{align}
\end{theorem}
\begin{proof} Observe that if we set $\lambda=1$ in \eqref{3.10} we obtain
\eqref{3.9}. Since $|q|<1$, taking the limit $\lambda \to \infty$, 
 setting $j=0$, and taking into account the definition of the 
 $q$-Bateman polynomials yields \eqref{3.11}.
Let us prove \eqref{3.10}. Setting $z=-2n-1-j$ in the 
$q$-Pasternack polynomial, and using their basic hypergeometric 
expansion, we have 
\begin{align*}
\sum_{n=0}^\infty {\mathscr B}^{j}_{m}(-2n-1-j;q)\,\frac{(q^\lambda;q)_n}
{(q;q)_n}\, t^n= & \sum_{n=0}^\infty \sum_{k=0}^n \frac{(q^{-m};q)_k
(q^{m+1};q)_{k}(q^{-n};q)_k}{(q;q)^2_k (q^{j+1};q)_k}\, q^{nk}
t^n\, \frac{(q^\lambda;q)_n}{(q;q)_n}\\[3mm]=
&\sum_{n=0}^\infty\sum_{k=0}^n \frac{(q^{-m};q)_k(q^{m+1};q)_{k}
 (-1)^kq^{\racion{k(k-1)}2} t^n (q^\lambda;q)_n}
{(q;q)^2_k(q^{j+1};q)_k(q;q)_{n-k}}\, 
\\[3mm]=& \sum_{k=0}^\infty\sum_{n=0}^\infty \frac{(q^{-m};q)_k
(q^{m+1};q)_{k}(-t)^k q^{\racion{k(k-1)}2}}{(q;q)^2_k(q^{j+1};q)_k}
\frac{t^n (q^\lambda;q)_{n+k}}{(q;q)_{n}}
\\[3mm] =&\frac{1}{(t;q)_{\lambda+k}} 
\sum_{k=0}^\infty \frac{(q^{-m};q)_k
(q^{m+1};q)_{k}(q^{\lambda};q)_{k}(-t)^k q^{\racion{k(k-1)}2}}{(q;q)_k
(q;q)_k(q^{j+1};q)_k} \\[3mm]
=& \frac 1{(t;q)_\lambda} {}_3\phi_3\left(\!\!
\begin{array}{c} q^{-m}, q^{m+1}, q^\lambda \\ q, q^{j+1}, t q 
\end{array} \!; q; t\right).
\end{align*} 
\end{proof}

Lemma \ref{lem1} is straightforward as a limiting 
case of expression \eqref{3.11}.
The result follows by setting $t\mapsto t(1-q)$, taking the 
limit $q\uparrow 1$, and taking into account the limiting relations
\[
\lim_{q\uparrow 1}\frac{(1-q)^n}{(q;q)_n}=\frac 1{n!},\qquad 
\lim_{q\uparrow 1}\frac{1}{(t(1-q);q)_\infty}=e^t, 
\]
and
\[
\lim_{q\uparrow 1} {}_2\phi_2\left(\!\!\begin{array}{c} q^{a}, q^{b}\\ q^c, q^d
\end{array} \!; q; t(q-1)\right)={}_2F_2\left(\!\!\begin{array}{c} a, b\\ c, d
\end{array} \!; t\right).
\]
In fact, we obtain a new generating function for the Pasternack and 
Bateman polynomials by taking the same limit $q\uparrow 1$.

\begin{corollary}
The following identities hold:
\[
\sum_{n=0}^\infty {\mathscr B}_{m}(-2n-1)\,t^n = 
\frac 1{1-t}\, 
P_m\left(\frac{1+t}{1-t}\right),
\]
\[
\sum_{n=0}^\infty {\mathscr B}_{m}(-2n-1)\,\frac{(\lambda)_n}
{n!}\, t^n=  \ \frac 1{(1-t)^\lambda} \, {}_3F_2\left(
\begin{array}{c} -m, m+1, \lambda \\ 1, 1\end{array}; \frac{-t}{1-t}\right),
\]
\[
\sum_{n=0}^\infty {\mathscr B}^{j}_{m}(-2n-1-j)\,t^n = 
\frac 1{1-t}\, {}_2F_1\left(\begin{array}{c} -m , m+1 \\ 
j+1 \end{array} ; \frac{-t}{1-t}\right),
\]
\[
\sum_{n=0}^\infty {\mathscr B}^{j}_{m}(-2n-1-j)\,\frac{(\lambda)_n}
{n!}\, t^n=  \ \frac 1{(1-t)^\lambda} \, {}_3F_2\left(
\begin{array}{c} -m, m+1, \lambda \\ 1, j+1\end{array}; \frac{-t}{1-t}\right),
\]
where $P_m$ is the Legendre polynomial.
\end{corollary}

\begin{theorem} \label{thm:3.3}For any $t\in \mathbb C$ small enough, the 
$q$-Pasternack  polynomials satisfy the following generating function:
\begin{align}
\label{3.4} \sum_{n=0}^\infty {\mathscr B}^m_{n}(z;q)\,t^n=
&\frac 1{1-t}\,{}_5\phi_5\left(\!\!\begin{array}{c} -q,q^{\racion12}, 
-q^{\racion12}, q^{\racion{z+m+1}2}, 0 \\ q^{m+1}, (qt)^{\racion12}, 
-(qt)^{\racion 12}, qt^{\racion12}, -qt^{\racion12}
\end{array}\!;q,t\right).
\end{align}
\end{theorem}
\begin{proof} Taking into account the expression of the 
basic hypergeometric series for these polynomials we have
\begin{align*}
\sum_{n=0}^\infty {\mathscr B}^m_{n}(z;q)\,t^n&= \sum_{n=0}^\infty 
\sum_{k=0}^n \frac{(q^{-n};q)_k(q^{n+1};q)_k(q^{\racion{z+m+1}
2};q)_k}{(q;q)^2_k(q^{m+1};q)_k}\,(q^kt)^n\\[3mm] =&
\sum_{n=0}^\infty \sum_{k=0}^n \frac{(q;q)_{n+k}(q^{\racion{z+m
+1}2};q)_k}{(q;q)_{n-k}(q;q)^2_k(q^{m+1};q)_k}(-1)^k 
q^{\racion{k(k-1)}2}t^n\\[3mm] =& \sum_{k=0}^\infty \sum_{n=0}^\infty 
\frac{(q;q)_{n+2k}(q^{\racion{z+m+1}2};q)_k}{(q;q)_{n}(q;q)^2_k
(q^{m+1};q)_k}(-1)^k q^{\racion{k(k-1)}2}t^{n+k}\\[3mm]
=& \sum_{k=0}^\infty \frac{(q;q)_{2k}
(q^{\racion{z+m+1}2};q)_k}{(q;q)^2_k(q^{m+1};q)_k}(-1)^k 
q^{\racion{k(k-1)}2}\,t^{k}\frac{(q^{2k+1}t;q)_\infty}{(t;q)_\infty}\\[3mm] =& 
\frac1{1-t}\sum_{k=0}^\infty \frac{(q;q)_{2k}
(q^{\racion{z+m+1}2};q)_k}{(q;q)^2_k(q^{m+1};q)_k(qt;q)_{2k}}(-1)^k 
q^{\racion{k(k-1)}2}\,t^{k}\\[3mm] =&\frac1{1-t}\sum_{k=0}^\infty 
\frac{(q^{\racion 12};q)_{k}(-q^{\racion 12};q)_{k}, (-q;q)_{k}(q^{\racion{z+m
+1}2};q)_k\, (-1)^k q^{\racion{k(k-1)}2}}{(q;q)_k(q^{m+1};q)_k((qt)^{\racion 
12};q)_{k}(-(qt)^{\racion 12};q)_{k}(qt^{\racion 12};q)_{k}(-qt^{\racion 
12};q)_{k}}\,t^{k} .
\end{align*}
\end{proof}

Observe that we obtain \eqref{3.2} by setting $m=0$ in \eqref{3.4}.
\begin{theorem}
The $q$-Sylvester polynomials satisfy the following generating functions:
\begin{align}
\label{3.6}\sum_{n=0}^\infty \varphi_{n}(z;q)\, t^n=& \frac {(q^z t; q)_\infty}
{(t,zt; q)_\infty},\\[3mm]
\label{3.7}\sum_{n=0}^\infty (q^\lambda;q)_n\,\varphi_{n}(z;q)\, t^n=&\frac1
{(zt;q)_\lambda}\,{}_2\phi_1\left(\!\!\begin{array}{c} q^\lambda, q^z\\ 
ztq^\lambda \end{array}\!;q,t\right),
\end{align}
where the principal values of $q^z$ and $q^\lambda$ are taken.
\end{theorem}
\begin{proof}
Let us prove the generating function \eqref{3.7} by using an analogous 
method as before, namely 
\begin{align*}
\sum_{n=0}^\infty (q^\lambda;q)_n\,\varphi_{n}(z;q)\, t^n=&
\sum_{n=0}^\infty \sum_{k=0}^n \frac{(q^{-n};q)_{k}(q^z;q)_k(-1)^k
q^{-\racion{k(k-1)}2}(q^\lambda;q)_n}{(q;q)_{n}(q;q)_k}\, q^{nk}
z^{n-k}t^n\\[3mm] =& \sum_{n=0}^\infty\sum_{k=0}^n \frac{(q^z;q)_{k}
(q^\lambda;q)_n}{(q;q)_{n-k}(q;q)_k}\, z^{n-k}t^n=\sum_{k=0}^\infty
\sum_{n=0}^\infty \frac{(q^z;q)_{k}(q^\lambda;q)_{n+k}}{(q;q)_{n}
(q;q)_k}\, z^{n}t^{n+k}\\[3mm]
=& \sum_{k=0}^\infty \frac{(q^z;q)_{k}(q^\lambda;q)_{k}}{(q;q)_k}\,t^{k}
\sum_{n=0}^\infty  \frac{(q^{\lambda+k};q)_n(zt)^{n}}{(q;q)_{n}}=
\sum_{k=0}^\infty \frac{(q^z;q)_{k}(q^\lambda;q)_{k}}{(q;q)_k
(zt;q)_{\lambda+k}}\,t^{k}\\[3mm] =& \frac 1{(zt;q)_\lambda}\sum_{k=0}^\infty 
\frac{(q^z;q)_{k}(q^\lambda;q)_{k}}{(q;q)_k(ztq^\lambda;q)_{k}}\,t^{k}.
\end{align*}
Since $|q|<1$, \eqref{3.6} follows from taking $\lambda \to \infty$ and 
applying \eqref{2.1}, both generating functions hold.
\end{proof}

\begin{theorem} For any $t\in \mathbb C$ small enough, the $q$-Ces\`aro 
polynomials satisfy the following generating function:
\begin{align}\label{3.8}
\sum_{n=0}^\infty g^{(s)}_{n}(z;q)\,t^n=& \frac 1{(1-t z q^s)(t;q)_{s+1}}.
\end{align}
\end{theorem}
\begin{proof}
Let us prove this by using \eqref{qCP} and some 
basic properties of the $q$-Pochhammer symbol:
\begin{align*}
\sum_{n=0}^\infty g^{(s)}_{n}(z;q)\,t^n =& \sum_{n=0}^\infty
\sum_{k=0}^n \left[\!\!\begin{array}{c}{k+s} \\ s \end{array}\!\!\right]_q
\, (q z^s)^{n-k} t^n=\sum_{n=0}^\infty\sum_{k=0}^n 
\left[\!\!\begin{array}{c}{n-k+s} \\ s \end{array}\!\!\right]_q
\, (q z^s)^{k} t^n\\[3mm]
=& \sum_{k=0}^\infty\sum_{n=0}^\infty
\left[\!\!\begin{array}{c}{n+s} \\ s \end{array}\!\!\right]_q
\, (t q z^s)^{k} t^n=\frac 1{(1-t z q^s)(t;q)_{s+1}}.
\end{align*}
\end{proof}
We leave it to the reader to check that by taking limit $q\uparrow1$, we 
recover the original expressions by using the limit identities we 
gave before Theorem \ref{thm:3.3}. In the future, it would be interesting to see 
if it is possible to use $q$-calculus 
to obtain $q$-analogues of the results obtained in \cite{bat1}.

\section*{Acknowledgments}
The author R. S. Costas-Santos acknowledges financial support by Direcci\'on 
General de Investigaci\'on, Ministerio de Econom\'ia y Com\-pe\-ti\-ti\-vi\-dad 
of Spain, grant MTM2015-65888-C4-2-P.
\bibliographystyle{elsarticle-num}

\begin{thebibliography}{10}
 \bibitem{bat1} Bateman, H.
\newblock Two systems of polynomials for the solution of Laplace integral
 equation. \newblock {\em Duke Math. Journal {\bf 2} (1936), 569--577}.
\bibitem{bat2} Bateman, H.
\newblock {\em Some properties of a certain set of polynomials}.
\newblock {\em T\^ohuko Math. Journal {\bf 37} (1933), 23--38}.
\bibitem{gara} Gasper, G.; Rahman, M.
\newblock {\em Basic Hypergeometric Series.}
\newblock Cambridge University Press, Cambridge, 2004.
\bibitem{kost} Koekoek, R.; Lesky, P. A; Swarttouw, R. F.
\newblock {\em Hypergeometric orthogonal polynomials and their
  {$q$}-analogues}.
  \newblock Springer Monographs in Mathematics. Springer-Verlag, Berlin, 2010.
  \newblock With a foreword by Tom H. Koornwinder.
\bibitem{koe} Koelink, H. T.
\newblock  On Jacobi and continuous Hahn polynomials.
\newblock {\em Proc. Amer. Math. Soc. {\bf 124}(3) (1996), 887--898.}
\bibitem{pas} Pasternack, S.
\newblock A generalisation of the polynomial $F_n(x)$.
\newblock {\em London, Edinburgh, Dublin Philosophical Magazine and
J. Science {\bf 28}, ser. 7 (1939), 209--226.}
\bibitem{rai} Rainville, E. D.
\newblock {\em  Special Functions.}
\newblock Macmillan, New York, 1960.
\bibitem{srma} Srivastava, H. M.; Manocha, H. L.
\newblock {\em A Treatise on Generating Functions.}
\newblock Wiley, New York, 1984.
\end{thebibliography}

\smallskip
\smallskip
\smallskip
\smallskip
\smallskip
\smallskip
\smallskip

\begin{minipage}{0.39\textwidth}
\begin{flushleft}
\begin{footnotesize}
Howard S. Cohl\\ 
Applied and Computational Mathematics Division \\
National Institute of Standards and Technology \\
Mission Viejo, CA 92694\\ USA  \\
howard.cohl@nist.gov
\end{footnotesize}
\end{flushleft}
\end{minipage}
\noindent\begin{minipage}{0.31\textwidth}
\begin{flushleft}
\begin{footnotesize}
Roberto S. Costas-Santos\\
Dpto. F\'isica y Matem\'aticas \\
Universidad de Alcal\'a \\
Alcal\'a de Henares, Madrid\\ SPAIN\\
rscosa@gmail.com
\end{footnotesize}
\end{flushleft} 
\end{minipage}
\begin{minipage}{0.25\textwidth}
\begin{flushleft}
\begin{footnotesize}
Tanay V. Wakhare\\
University of Maryland \\ College Park, MD 20742\\ USA\\
twakhare@gmail.com
\end{footnotesize}
\end{flushleft} 
\end{minipage}
\end{document}